\documentclass{amsart}
\usepackage[margin=3cm]{geometry}
\usepackage{amsmath,amsfonts,amssymb,amsthm}
\usepackage{graphics}
\usepackage{epsfig, esint}
\usepackage{graphics,color}
\usepackage{paralist}

\RequirePackage[colorlinks,citecolor=blue,urlcolor=blue]{hyperref}

   \definecolor{labelkey}{gray}{.8}
   \definecolor{refkey}{gray}{.8}
\providecommand{\R}{\mathbb{R}}

\providecommand{\divv}{{\rm {div}}}

\newcommand{\e}{\varepsilon}

\newcommand{\step}[1]{\medskip\noindent\textbf{Step #1. }}

\newcommand{\ignore}[1]{}

\newtheorem{definition}{Definition}
\newtheorem{proposition}{Proposition}
\newtheorem{theorem}{Theorem}

\newtheorem{lemma}{Lemma}
\newtheorem{corollary}{Corollary}
\newtheorem{assumption}{Assumption}

\author[M. Sch\"affner]{Mathias Sch\"affner}
\address[Mathias Sch\"affner]{Institut f\"ur Mathematik, MLU Halle-Wittenberg, Theodor-Lieser-Stra\ss e 5, 06120 Halle (Saale), Germany}
\email{mathias.schaeffner@mathematik.uni-halle.de}

\title[Local Lipschitz-bounds for scalar integral functionals]{Lipschitz bounds for nonuniformly elliptic integral functionals in the plane} 

\begin{document}
\maketitle

%\centerline{\LARGE{NOT FOR DISTRIBUTION}}

\begin{abstract}
We study local regularity properties of local minimizer of scalar integral functionals with controlled $(p,q)$-growth in the two-dimensional plane. We establish Lipschitz continuity for local minimizer under the condition $1<p\leq q<\infty$ with $q<3p$ which improve upon the classical results valid in the regime $q<2p$. Along the way, we establish an $L^\infty$-$L^2$-estimate for solutions of linear uniformly elliptic equations in the plane which is optimal with respect to the ellipticity contrast of the coefficients.% optimal dependence with respect to the ellipticity ratio of the coefficients.
\end{abstract}

\section{Introduction}

In this short note, we revisit the question of Lipschitz-regularity for local minimizers of integral functionals of the form
\begin{equation}\label{eq:int}
w\mapsto \mathcal F(w;\Omega):=\int_\Omega F(\nabla w)\,dx,
\end{equation}
where the integrand $F$ satisfies so-called $(p,q)$-growth conditions which are described in the following

\begin{assumption}\label{ass}
Let $0<\nu\leq \Lambda<\infty$, $1<p\leq q<\infty$ and $\mu\in[0,1]$ be given. Suppose that $F:\R^2\to[0,\infty)$ is convex, locally $C^2$-regular in $\R^n\setminus\{0\}$ and satisfies
\begin{equation}\label{ass:Fpq}
\begin{cases}
\nu(\mu^2+|z|^2)^\frac{p}2\leq F(z)\leq \Lambda(\mu^2+|z|^2)^\frac{q}2+\Lambda (\mu^2+|z|^2)^\frac{p}2\\
|\partial^2 F(z)|\leq \Lambda(\mu^2+|z|^2)^\frac{q-2}2+\Lambda (\mu^2+|z|^2)^\frac{p-2}2,\\
\nu(\mu^2+|z|^2)^\frac{p-2}2|\xi|^2\leq \langle \partial^2 F(z)\xi,\xi\rangle,
\end{cases}
\end{equation}
for every choice of $z,\xi\in\R^n$ with $|z|>0$.
\end{assumption}
We recall that regularity theory for local minimizer of \eqref{eq:int} under Assumption~\ref{ass} with $1<p=q$ is classical, see e.g.\ \cite{Giu}. A systematic regularity theory in the case $p<q$ starts with the seminal contributions of Marcellini \cite{Mar91,Mar96} and became an active field of research (see e.g.\ \cite{MR21,Min24} or the introduction of \cite{Mar21} for recent overviews on the literature).
  
Before we state our main result, we recall a standard notion of local minimality in the context of integral functionals with $(p,q)$-growth 
\begin{definition}\label{def:localmin}
We call $u\in W_{\rm loc}^{1,1}(\Omega)$ a local minimizer of  $\mathcal F$ given in \eqref{eq:int} if and only if $F(\nabla u)\in L_{\rm loc}^1(\Omega)$ and for every $\Omega'\Subset\Omega$ holds $\mathcal F(u,\Omega')\leq \mathcal F(u+\varphi,\Omega')$ for every $\varphi\in W_0^{1,1}(\Omega')$.

%for every open set $\Omega'\Subset\Omega$ it holds $F(\nabla u)\in L_{\rm loc}^1(\Omega')$
%the following is true: 
%%
%\begin{equation*}
% \mathcal F(u,\Omega')<\infty
%\end{equation*}
%%
%and
%
%\begin{equation*}
% F(\nabla u)\in L_{\rm loc}^1(\Omega)\qquad\mbox{and}\qquad \mathcal F(u,\Omega')\leq \mathcal F(u+\varphi,\Omega')
%\end{equation*}
%%
%for any $\varphi\in W^{1,1}(\Omega)$ satisfying ${\rm supp}\;\varphi\Subset \Omega'$.
\end{definition}
The main result of the present paper is

\begin{theorem}\label{T:1}
Let $\Omega\subset\R^2$ be an open bounded domain and suppose Assumption~\ref{ass} is satisfied with $1<p\leq q<\infty$ such that
\begin{equation}\label{eq:pq}
q <3p.%\qquad\mbox{and}\qquad f\in L^{n,1}(\Omega).
\end{equation}
Let $u\in W_{\rm loc}^{1,1}(\Omega)$ be a local minimizer of the functional $\mathcal F$ given in \eqref{eq:int}. Then $\nabla u$ is locally bounded in $\Omega$. Moreover,  there exists $c=c(\Lambda,\nu,p,q)\in[1,\infty)$ such that for all $B\Subset \Omega$ it holds
\begin{align}\label{est:T1}
\|\nabla u\|_{L^\infty(\frac12 B)}\leq& c\biggl(\fint_B F(\nabla u)\,dx\biggr)^\frac1p+c\biggl(\fint_B F(\nabla u)\,dx\biggr)^{\frac{2}{3p-q}}.
\end{align}
\end{theorem}
Theorem~\ref{T:1} improves upon the by now classical findings of Marcellini \cite{Mar96}, where Lipschitz-continuity is proven under the relation $q<2p$. More precisely, Marcellini considered integral functionals in arbitrary dimensions $n\geq2$ and proved local Lipschitz continuity for minimizer of \eqref{eq:int} under the restriction
\begin{equation}\label{eq:pqold}
\frac{q}p<1+\frac2n,
\end{equation}
that is $q<2p$ for $n=2$. Recently, the condition \eqref{eq:pqold} for Lipschitz-continuity was relaxed in joint works with Bella \cite{BS19c,BS24} to 
\begin{equation}\label{eq:assq2}
\frac{q}{p}<1+\frac2{n-1}
\end{equation}
in dimensions $n\geq3$, but the methods of \cite{BS19c,BS24} (that are based on \cite{BS19a}) do not give any improvement compared to Marcellini's result in dimension $n=2$.  Combining Theorem~\ref{T:1} with the findings of \cite{BS19c,BS24}, we have  
%\begin{remark}
%Theorem~\ref{T:1} improves upon the by now classical findings of Marcellini \cite{Mar91}, where Lipschitz-continuity is proven under the relation $q<2p$ and extend naturally the results of \cite{BS19c,BS24}, where Lipschitz continuity is proven in dimension $n\geq3$ under the relation $q<(1+\frac2{n-1})p$ (which gives $q<(1+2)p=3p$ for $n=2$)
%\end{remark}
\begin{theorem}\label{T:2}
Let $\Omega\subset \R^n$, $n\geq2$ and suppose Assumption~\ref{ass} is satisfied with $1<p\leq q<\infty$ such that \eqref{eq:assq2}. Let $u\in W_{\rm loc}^{1,1}(\Omega)$ be a local minimizer of the functional $\mathcal F$ given in \eqref{eq:int}. Then, $u\in W_{\rm loc}^{1,\infty}(\Omega)$. 
\end{theorem}
%
%
%To the best of my knowledge \eqref{eq:assq2} is currently the best bound yielding the conclusion in Theorem~\ref{T:2}, see \cite{DKK} for an improvement on the ratio $q/p$ under stronger growth conditions (which are satisfied by several examples). 

Optimality of condition \eqref{eq:pq} in Theorem~\ref{T:1} is not clear. We note that the results of \cite{CS23,HS19} imply that in the situation of Theorem~\ref{T:1} local minimizer are locally bounded for every choice of $1\leq p\leq q<\infty$, and recall that by Marcellini's counterexample, see \cite[Theorem 6.1]{Mar91}, this is in general not true in dimensions $n\geq3$ (and exponents $1<p<n-1$ and $q<\infty$ sufficiently large).

\smallskip

The strategy to prove Theorem~\ref{T:1} is similar to \cite{BM20,BS24} and relies on careful estimates for certain \textit{uniformly elliptic} problems. Indeed, one of the main technical achievement in \cite{BS24} is to establish a local $L^\infty$-$L^2$-estimate for solutions of linear uniformly elliptic equations which is essentially optimal with respect to the ellipticity contrast in dimension $n\geq3$, see \cite[Proposition 1]{BS24}.  Here, we provide the analogous statement in dimension $n=2$, namely
\begin{proposition}\label{P:2}
Let $B=B_R(x_0)\subset\R^2$. There exists $c<\infty$, such that the following is true. Let $0<\nu\leq \lambda<\infty$ and suppose $a\in L^\infty(B;\R^{2\times 2})$ satisfies that $a(x)$ is symmetric for almost every $x\in B$ and is uniformly elliptic in the sense
\begin{equation}\label{def:ellipt}
\nu|z|^2\leq a(x)z\cdot z\leq \Lambda|z|^2\quad\mbox{for every $z\in\R^2$ and almost every $x\in B$.}
\end{equation}
Let $v\in W^{1,2}(B)$ be a subsolution, that is it satisfies
\begin{equation}\label{def:subsolution}
\int_{B}a\nabla v\cdot\nabla \varphi\leq0\qquad\mbox{for all $\varphi\in C_c^1(B)$ with $\varphi\geq0$.}
\end{equation}
Then,
\begin{equation}\label{est:P1}
\sup_{\frac12 B} v\leq c\biggl(\frac{\Lambda}{\nu}\biggr)^\frac14\biggl(\fint_{B}(v_+)^2\,dx\biggr)^\frac12,
\end{equation}
where $v_+:=\max\{v,0\}$. Moreover, the exponent $\frac{1}4$ in \eqref{est:P1} is optimal.
\end{proposition}

The proof of Proposition~\ref{P:2} is elementary and only relies on the maximum principle, a suitable one dimensional Poincar\'e inequality and the Caccioppoli inequality, see Section~\ref{sec:proofP2}. However, Proposition~\ref{P:2} (or rather its proof) is crucial for the proof of Theorem~\ref{T:1}. To illustrate the relation between Proposition~\ref{P:2} and Theorem~\ref{T:1}, we consider the case that $F$ satisfies Assumption~\ref{ass} with $p=2<q$: Let $u$ be a local minimizer of $\mathcal F(\cdot,B_1)$ and assume that $u$ is smooth. Standard arguments yield
\begin{equation}\label{intro:lineq}
\divv (A(x)\nabla \partial_i u)=0\qquad\mbox{where}\qquad A:=\partial^2F(\nabla u).
\end{equation}
Hence, $\partial_i u$ satisfies a linear elliptic equation where $A$ satisfies (using \eqref{ass:Fpq})
$$
|z|^2\lesssim Az\cdot z\lesssim  (1+\|\nabla u\|_{L^\infty(B_1)}^{q-2})|z|^2. 
$$
Hence, Proposition~\ref{P:2} (applied to $\partial_1 u $ and $\partial_2 u$) yields
\begin{equation}\label{intro:linftyl2}
\|\nabla u\|_{L^\infty(B_\frac12)}\lesssim (1+ \|\nabla u\|_{L^\infty(B_1)}^\frac{q-2}4)\|\nabla u\|_{L^2(B_1)}.
\end{equation}
Appealing to some well-known iteration arguments, it is possible to absorb the prefactor $ \|\nabla u\|_{L^\infty(B_1)}^\frac{q-2}4$ on the right-hand side in \eqref{intro:linftyl2} provided that $\frac{q-2}4<1$, which yields the restriction $q<6$ coinciding with \eqref{eq:pq} for $p=2$. Once an a priori Lipschitz estimate for smooth minimizer is established, the proof of Theorem~\ref{T:1} follows by a careful regularization and approximation procedure -- for this we closely follow related arguments from \cite{DM23,DM24}.

\smallskip

Next, we discuss further related literature: Firstly, we mention the recent paper \cite{DKK}, where a variation of $p,q$-growth conditions is considered and local Lipschitz-continuity is proven in dimension two for all exponents $2\leq p<q$. In \cite{DKK}, the assumptions \eqref{ass:Fpq} are strengthened by assuming an additional relation between the growth of $\partial^2 F$ and $\partial F$. Let us now discuss some possible extensions of Theorem~\ref{T:1}: Here we assume polynomial growth conditions with $p>1$. There are several results (and recent interest) for integrands satisfying more general growth conditions and in particular allowing for nearly linear growth, see e.g.\ \cite{BR20,DM23b,DP24,EMMP22,FM,Mar96}. We expect that the method to prove Theorem~\ref{T:1} can be combined with the techniques of e.g.\ \cite{EMMP22,FM,Mar96} to improve (in parts) these results in $n=2$ (and for $n\geq 3$ with the methods of \cite{BS24}). As mentioned above, we rely in the proof of Theorem~\ref{T:1} on the maximum principle and thus the argument cannot be used for vectorial problems. It would be interesting to find an argument which is applicable for systems to improve higher integrability results, see e.g.\ \cite{CKP14,ELM99,S21,Koch}, for vectorial integral functionals in dimension $n=2$, see \cite{BF03} for a full regularity result under Assumption~\ref{ass} with $n=2$ and $q<2p$. 

\smallskip

In the above discussion, we mostly ignored the rich regularity theory for non-autonomous non-uniformly elliptic integral functionals, we refer again to the recent overviews \cite{MR21,Min24} and mention some recent papers \cite{BCM18,CMMP21,CFK20,DM21,DM23,EMM19,EP23,HO21} and \cite{BDS20,BGS21,Borowski2023} for related results about the Lavrentiev phenomena.

\section{Preliminaries}

We begin with an elementary one dimensional interpolation inequality. 
\begin{lemma}\label{L:1dinterpolation} Let $I\subset \R$ be a bounded interval. Then, it holds 
\begin{equation}\label{est:1dinterpolation}
\forall u\in H^1(I):\qquad \|u- (u)_I\|_{L^\infty(I)}\leq \sqrt{2}\|u-(u)_I\|_{L^2(I)}^\frac12\|u'\|_{L^2(I)}^\frac12,
\end{equation}
where $(u)_I:=\fint_Iu(s)\,ds$.
\end{lemma}

\begin{proof}
By density it suffices to show \eqref{est:1dinterpolation} for $u\in C^1(I)$ and without loss of generality we assume $(u)_I=0$. By the integral mean value theorem there exists $y\in I$ such that $u(y)=(u)_I=0$. By the fundamental theorem of calculus, we have for every $x\in I$
\begin{equation*}
u^2(x)=2\biggl|\int_y^xu(s)u'(s)\,ds\biggr|\leq2 \|u\|_{L^2(I)}\|u'\|_{L^2(I)}.
\end{equation*}
We obtain \eqref{est:1dinterpolation} by applying $\sup_{x\in I}$ and taking the square-root of the resulting inequality. 
\end{proof}

In the following, we use Lemma~\ref{L:1dinterpolation} in the form below
\begin{corollary}\label{C:pickslice}
There exists $c<\infty$ such that the following is true: For every $u\in H^1(B_1)$ and $\frac14\leq \rho<\sigma<1$, it holds
\begin{equation}
\inf_{r\in(\rho,\sigma)}\|u\|_{L^\infty(S_r)}\leq \frac{c}{(\sigma-\rho)^\frac12}\biggl(\|u\|_{L^2(B_\sigma\setminus B_\rho)}+\|u\|_{L^2(B_\sigma\setminus B_\rho)}^\frac12\|\nabla u\|_{L^2(B_\sigma\setminus B_\rho)}^\frac12\biggr)
\end{equation}
\end{corollary}

\begin{proof}
Let $u\in H^1(B_1)$. For $r\in[0,1]$, we define $u_r:\partial B_1\to\R$ as $u_r(z)=u(rz)$ for $z\in \partial B_1$. Clearly, $u_r$ is weakly differentiable for almost every $r\in[0,1]$ and by Fubini's theorem there exists a Nullset $N\subset [0,1]$ such that $u_r\in H^1(\partial B_1,\R)$ for $r\in[0,1]\setminus N$. We have
\begin{align*}
(\sigma-\rho)^\frac12\inf_{r\in(\rho,\sigma)}\|u\|_{L^\infty(\partial B_r)}\leq& \biggl(\int_{(\rho,\sigma)\setminus N}\|u\|_{L^\infty(\partial B_r)}^2\,dr\biggr)^\frac12=\biggl(\int_{(\rho,\sigma)\setminus N}\|u_r\|_{L^\infty(\partial B_1)}^2\,dr\biggr)^\frac12\\
\leq&\biggl(\int_{(\rho,\sigma)\setminus N}\|u_r-(u_r)_{\partial B_1}\|_{L^\infty(\partial B_1)}^2\,dr\biggr)^\frac12+\biggl(\int_{\rho}^\sigma\biggl(\int_{\partial B_1}|u(rz)|\,dz\biggr)^2\,dr\biggr)^\frac12\\
\stackrel{\eqref{est:1dinterpolation}}\lesssim&\biggl(\int_{(\rho,\sigma)\setminus N}r^\frac12\|u_r\|_{L^2(\partial B_1)}\|\nabla_{\rm tan}u_r\|_{L^2(\partial B_1)}\,dr\biggr)^\frac12+\|u\|_{L^2(B_\sigma\setminus B_\rho)}\\
\lesssim&\|u\|_{L^2(B_\sigma\setminus B_\rho)}^\frac12\|\nabla u\|_{L^2(B_\sigma\setminus B_\rho)}^\frac12+\|u\|_{L^2(B_\sigma\setminus B_\rho)}
\end{align*}
and the claim follows.
\end{proof}

Finally, we recall here the following classical iteration lemma
\begin{lemma}[Lemma~6.1, \cite{Giu}]\label{L:holefilling}
Let $Z(t)$ be a bounded non-negative function in the interval $[\rho,\sigma]$. Assume that for every $\rho\leq s<t\leq \sigma$ it holds
\begin{equation*}
 Z(s)\leq \theta Z(t)+(t-s)^{-\alpha} A+B,
\end{equation*}
with $A,B\geq0$, $\alpha>0$ and $\theta\in[0,1)$. Then, there exists $c=c(\alpha,\theta)\in[1,\infty)$ such that
\begin{equation*}
 Z(s)\leq c((t-s)^{-\alpha} A+B).
\end{equation*}
\end{lemma}

\section{Proof of Proposition~\ref{P:2}}\label{sec:proofP2}

The proof of estimate \eqref{est:P1} in Proposition~\ref{P:2} follows by a combination of Corollary~\ref{C:pickslice} and the maximum principle, see e.g.\ \cite[Proposition 3.4]{BS19a} for related arguments for nonuniformly linear elliptic equations.

\begin{proof}[Proof of Proposition~\ref{P:2}]
Without loss of generality, we only consider $B=B_1(0)$ and $1=\nu\leq \Lambda$. The general claim follows by standard scaling and translation arguments and multiplying the equation by $\nu^{-1}$.

%Then the positive part $v_+$ of $v$ is still a subsolution.

\step 1 Let $v\in H^1(B)$ be a subsolutions in the sense of \eqref{def:subsolution}. We claim that
\begin{equation}\label{P1:step1:est0}
\sup_{B_\frac12} v\lesssim\|v_+\|_{L^2(B_\frac34\setminus B_\frac12)}+\|v_+\|_{L^2(B_\frac34\setminus B_\frac12)}^\frac12\|\nabla v_+\|_{L^2(B_\frac34\setminus B_\frac12)}^\frac12.
\end{equation} 
The maximum principle for subsolution, see e.g. \cite[Theorem 8.1]{GT}, yields %for every $r\in (\frac12,\frac34)$ that
\begin{align*}%\label{P1:step1:est1}
\forall r\in[\tfrac12,\tfrac34]:\quad \sup_{B_\frac12}v\leq \sup_{B_r}v \leq \|v_+\|_{L^\infty(\partial B_r)}%\notag\\
%\lesssim&\|u_+\|_{L^2(\partial B_r)}+\|u_+\|_{L^2(\partial B_r)}^\frac12\|\nabla_{\rm tan}u_+\|_{L^2(\partial B_r)}^\frac12.
\end{align*} 
and the claim \eqref{P1:step1:est0} follows from Corollary~\ref{C:pickslice}.

\step 2 Proof of estimate \eqref{est:P1}. Let $v\in H^1(B)$ be as in Step~1. By standard density arguments, we can choose $\varphi=v_+\eta^2$ with $\eta\in C_c^1(B_1)$ in \eqref{def:subsolution} and obtain (after standard manipulations)
\begin{equation*}
\int_{B}\eta^2|\nabla v_+|^2\,dx\stackrel{\eqref{def:ellipt}}\leq\int_{B}\eta^2 a\nabla v_+\cdot\nabla v_+\,dx\leq 4\int_B v_+^2a\nabla \eta\cdot\nabla \eta\,dx\stackrel{\eqref{def:ellipt}}\leq 4\Lambda \int_B v_+^2|\nabla \eta|^2\,dx.
\end{equation*}
Choosing the obvious $\eta$, we find
\begin{equation}\label{P1:step2:est1}
\|\nabla v_+\|_{L^2(B_\frac34)}\lesssim \Lambda^\frac12\| v_+\|_{L^2(B_1)}.
\end{equation} 
Inserting \eqref{P1:step2:est1} into \eqref{P1:step1:est0} yields
\begin{equation*}%\label{P1:step1:est0}
\sup_{B_\frac12} v\stackrel{\eqref{P1:step1:est0}}\lesssim\|v_+\|_{L^2(B_\frac34)}+\|v_+\|_{L^2(B_\frac34)}^\frac12\|\nabla v_+\|_{L^2(B_\frac34)}^\frac12\stackrel{\eqref{P1:step2:est1}}\lesssim (1+\Lambda^\frac14)\|v_+\|_{L^2(B)}
\end{equation*} 
which yields the claimed estimate \eqref{est:P1} (recall $\Lambda\geq1=\nu$).

\step 3 Optimality. This follows by the same example used in \cite[Remark~3.2]{BS24} to show (almost) optimality of a related statement for $n\geq3$. We briefly recall the example (which is inspired by \cite{connor}).

Consider $v(x):=1+x_2^2-\Lambda x_1^2$ which clearly satisfies
\begin{equation}\label{eq:counterexample}
-\nabla\cdot a\nabla v=0\qquad\mbox{where}\qquad a:={\rm diag}(1,\Lambda).
\end{equation}
Obviously, we have $\sup_{B_\frac12}\, v\geq v(0)\geq1$ and
\begin{align*}
\int_{B_1}(v_+)^2\,dx\leq& 2\int_{-1}^1\int_{0}^{\Lambda^{-1/2}(1+x_2^2)^{1/2}}(1+x_2^2)^2-2(1+x_2^2)\Lambda x_1^2+\Lambda^2x_1^4\,dx_1\,dx_2\\
=&2\Lambda^{-1/2}(1-\frac23+\frac15)\int_{-1}^1(1+x_2^2)^{\frac{5}2}\,dx_2\\
\leq& \frac{32}{15} \cdot 2^\frac{5}2 \Lambda^{-1/2}.
\end{align*}
In particular, we have
$$
\frac{\sup_{B_\frac12}v}{\|v_+\|_{L^2(B_1)}}\geq \biggl(\frac{32}{15} \cdot 2^\frac{5}2 \biggr)^{-\frac12}\Lambda^{\frac{1}{4}},
$$
which shows that the exponent $\frac14$ in \eqref{est:P1} is sharp. 
\end{proof}

\section{A priori estimate for regularized problems}

In this section, we derive Lipschitz estimates for minimizers of certain regularized problems. We introduce the notation
\begin{equation}\label{def:Hmu}
H_\mu(z):=\mu^2+|z|^2,\qquad E_\mu(z):=\frac1p\biggl(H_\mu(z)^\frac{p}2-\mu^p\biggr)
\end{equation}
and suppose in this section
\begin{assumption}\label{ass:regF}
Let $0<\nu\leq \Lambda<\infty$, $\tilde \nu>0$, $1<p\leq q<\infty$ and $\mu\in(0,2]$ be given. Suppose that $F:\R^2\to[0,\infty)$ is convex, locally $C^2$-regular and satisfies for every choice of $z,\xi\in\R^n$
\begin{equation}\label{ass:Fpqreg}
\begin{cases}
\nu(H_\mu(z))^\frac{p}2+\tilde \nu (H_\mu(z))^\frac{q}2\leq F(z)\leq \Lambda(H_\mu(z))^\frac{q}2+\Lambda (H_\mu(z))^\frac{p}2\\
|\partial^2 F(z)|\leq \Lambda(H_\mu(z))^\frac{q-2}2+\Lambda (H_\mu(z))^\frac{p-2}2,\\
\left(\nu(H_\mu(z))^\frac{p-2}2+\tilde \nu (H_\mu(z))^\frac{q-2}2\right)|\xi|^2\leq \langle \partial^2 F(z)\xi,\xi\rangle.
\end{cases}
%\begin{cases}
%\nu(\mu^2+|z|^2)^\frac{p}2+\tilde \nu (\mu^2+|z^2)^\frac{q}2\leq F(z)\leq \Lambda(\mu^2+|z|^2)^\frac{q}2+\Lambda (\mu^2+|z|^2)^\frac{p}2\\
%|\partial^2 F(z)|\leq \Lambda(\mu^2+|z|^2)^\frac{q-2}2+\Lambda (\mu^2+|z|^2)^\frac{p-2}2,\\
%(\nu(\mu^2+|z|^2)^\frac{p-2}2+\tilde \nu (\mu^2+|z^2)^\frac{q-2}2)|\xi|^2\leq \langle \partial^2 F(z)\xi,\xi\rangle.
%\end{cases}
\end{equation}
\end{assumption}
In contrast to Assumptions~\ref{ass}, Assumptions~\ref{ass:regF} implies that $F$ satisfies standard $q$-growth conditions. In particular, classical regularity theory applies to local minimizer of \eqref{eq:int} under Assumption~\ref{ass:regF}. In the following lemma we provide a Lipschitz estimate for local minimizer of \eqref{eq:int} under Assumption~\ref{ass:regF} which is independent of the parameter $\tilde \nu$.

%Let us briefly compare Assumptions~\ref{ass} and \ref{ass:regF}: In Assumption~\ref{ass:regF} the integrands $F$ are more regular and satisfy standard $q$-growth conditions. In particular, classical regularity theory applies to minimizer of \eqref{eq:int} under Assumption~\ref{ass:regF}. In the following lemma we provide a Lipschitz estimate for local minimizer of \eqref{eq:int} under Assumption~\ref{ass:regF} which is independent of the parameter $\tilde \nu$.

\begin{lemma}\label{L:apriorilipschitz}
Let $\Omega\subset\R^2$ be a bounded domain and suppose Assumption~\ref{ass:regF} is satisfied with $1<p\leq q<3p<\infty$. Let $u\in W_{\rm loc}^{1,1}(\Omega)$ be a local minimizer of the functional $\mathcal F$ given in \eqref{eq:int}. Then, there exists $c=c(\Lambda,\nu,p,q)\in[1,\infty)$ such that for all $B\Subset \Omega$ it holds
\begin{align}\label{claim:L:apriorilipschitz}
\|\nabla u\|_{L^\infty(\frac12B)}^p\leq& c\fint_B H_\mu(|\nabla u|)^\frac{p}2\,dx+c\biggl(\fint_B H_\mu(|\nabla u|)^\frac{p}2\,dx\biggr)^{\frac2{3p-q}}.
\end{align}
\end{lemma}

%\begin{remark}\label{rem:gamma}
%Note that $1<p\leq q$ implies that $\gamma$ defined in \eqref{def:gammavartheta} is positive. Moreover, relation $q<3p$ implies $\gamma<1$. %Indeed, we have 
%%\begin{align*}
%%\frac{q}p<3\quad\Rightarrow&\quad \gamma=\frac14\frac{q}{p}+\frac14<\frac34+\frac14=1.
%%\end{align*}
%%
%\end{remark}

\begin{proof}[Proof of Lemma~\ref{L:apriorilipschitz}] Throughout the proof we write $\lesssim$ if $\leq$ holds up to a multiplicative constant depending only on $\Lambda,\nu,p$ and $q$. By standard scaling and translation arguments it suffices to consider the case $B=B_1\Subset \Omega$.

\step 0 Preliminaries. By Assumption~\ref{ass:regF}, the integrand $F$ satisfies standard $q$-growth. Hence, by classical elliptic regularity, we have that $u\in W^{2,2}_{\rm loc}(\Omega)\cap C_{\rm loc}^{1,\alpha}(\Omega)$ for some $\alpha\in(0,1)$,  see e.g.\ \cite{Giu}, and in particular $u$ solves the corresponding Euler-Lagrange equation
\begin{equation}\label{L:aprio:euler}
\nabla \cdot  (\partial F(\nabla u))=0\qquad\mbox{in $\Omega$}.
\end{equation}
Moreover, we can differentiate the Euler-Lagrange equation \eqref{L:aprio:euler} and obtain that $\partial_s u\in H_{\rm loc}^1(\Omega)$ solves 
\begin{equation}\label{L:aprio:eulerdiffesp}
\nabla \cdot  (\partial^2 F(\nabla u)\nabla \partial_s u)=0\qquad\mbox{in $\Omega$}.
\end{equation}
Finally, we state a crucial Caccioppoli inequality: There exists $c=c(\Lambda,\nu,p,q)\in[1,\infty)$ such that 
\begin{align}\label{eq:caccioaprio}
\int_{\frac12\tilde B}|\nabla E_\mu(\nabla u)|^2\,dx\leq& c(1+\|\nabla u\|_{L^\infty(\tilde B)}^{q-p})|\tilde B|^{-1}\int_{\tilde B} |E_\mu(\nabla u)|^2\,dx,
\end{align}
for all balls $\tilde B\subset B$ (recall $E_\mu(z)$ is defined in \eqref{def:Hmu}). Note that the $c$ in \eqref{eq:caccioaprio} is independent of $\tilde \nu$. Estimate \eqref{eq:caccioaprio} follows directly from \cite[Lemma 5.1]{DM23} (which is based on a more general result derived in \cite{BM20}) applied to \eqref{L:aprio:euler} with $A_0=\partial F$. Indeed, the symmetry of $\partial A_0=\partial^2 F$ is obvious and the estimates on $A_0$ and $\partial A_0$ stated in \cite[equation (5.16)]{DM23} follow from Assumption~\ref{ass:regF}, where we use that the upper bound on $F$ in \eqref{ass:Fpqreg} in combination with the convexity of $F$ yield
\begin{equation}\label{est:partialF}
|\partial F(z)|\leq \Lambda'(\mu^2+|z|^2)^\frac{q-1}2+\Lambda' (\mu^2+|z|^2)^\frac{p-1}2,
\end{equation}
with $\Lambda'=\Lambda'(\Lambda,p,q)$.

\step 1 In this step, we suppose $B=B_1\Subset \Omega$ and prove 
\begin{align}\label{est:keyapprio}
\|\nabla u\|_{L^\infty(B_\frac14)}^p
\lesssim&(\mu^\frac{p}2+\|\nabla u\|_{L^\infty(B_1)}^{\gamma p} + \|\nabla u\|_{L^\infty(B_1)}^\frac{p}2) \|H_\mu(|\nabla u|)^\frac{p}2\|_{L^1(B_1)}^\frac12+\mu^p,
\end{align}
where $\gamma:=\frac{q+p}{4p}$. In view of Step~0, we have $u\in W^{1,\infty}(B)$ and thus by Assumption~\ref{ass:regF} the matrix field $\partial^2 F(\nabla u)$ is uniformly elliptic in $B$. Hence, the maximum principle, see e.g. \cite[Theorem 8.1]{GT}, applied to equation \eqref{L:aprio:eulerdiffesp} yields for $s\in\{1,2\}$ and all $r\in(\frac14,\frac12)$ that
\begin{align*}
\|\partial_s u\|_{L^\infty(B_\frac14)}^p\leq \|\partial_s u\|_{L^\infty(\partial B_r)}^p\leq\|\nabla  u\|_{L^\infty(\partial B_r)}^p\leq p(\|E_\mu(|\nabla u|)\|_{L^\infty(\partial B_r)}+\mu^p)
\end{align*}
(where we use $|t|^p\leq (\mu^2+t^2)^\frac{p}2\leq p(E_\mu(t)+\mu^p)$). Thus we obtain with help of Corollary~\ref{C:pickslice} and \eqref{eq:caccioaprio}
\begin{align*}
\|\nabla u\|_{L^\infty(B_\frac14)}^p\lesssim& \|E_\mu(|\nabla u|)\|_{L^2(B_\frac12)}+\|E_\mu(|\nabla u|)\|_{L^2(B_\frac12)}^\frac1{2}\|\nabla E_\mu(|\nabla u|)\|_{L^2( B_\frac12)}^\frac1{2}+\mu^p\\
\lesssim&  (1+\|\nabla u\|_{L^\infty(B_1)}^\frac{q-p}4)\|E_\mu(|\nabla u|)\|_{L^2(B_1)}+\mu^p.
\end{align*}
The claimed estimate \eqref{est:keyapprio} follows from the elementary interpolation inequality $\|\cdot\|_{L^2}\leq (\|\cdot\|_{L^\infty}\|\cdot\|_{L^1})^\frac12$,
$$\|E_\mu(|\nabla u|)\|_{L^\infty(B_1)}^\frac12\leq \frac1{\sqrt{p}}(\mu^2+\|\nabla u\|_{L^\infty(B_1)}^2)^\frac{p}4\lesssim \mu^\frac{p}2+\|\nabla u\|_{L^\infty(B_1)}^\frac{p}2,$$
$E_\mu(t)\leq H_\mu(t)^\frac{p}2$ and $\frac{q-p}4+\frac{p}2=\frac{q+p}4=p\gamma$.

\step 2 Conclusion. Appealing to standard scaling and covering arguments, we deduce from Step~1 the following: For every $x_0\in \Omega$ and $0<\rho<\sigma$ satisfying $B_\sigma(x_0)\Subset \Omega$ it holds
\begin{align}\label{est:keyapprio1}
\|\nabla u\|_{L^\infty(B_\rho(x_0))}^p\lesssim&\biggl(\mu^\frac{p}2+\|\nabla u\|_{L^\infty(B_\sigma(x_0))}^{\gamma p}+\|\nabla u\|_{L^\infty(B_\sigma(x_0))}^\frac{p}2\biggr)\frac{\|H_\mu(|\nabla u|)^\frac{p}2\|_{L^1(B_\sigma(x_0))}^\frac12}{\sigma-\rho}+\mu^p.
\end{align}
From estimate \eqref{est:keyapprio1} in combination with Young inequality and assumption $1<q<3p$ in the form $0<\gamma=\frac{q+p}{4p}<\frac{4p}{4p}=1$, we obtain the existence of $c=c(\nu,\Lambda,p,q)\in[1,\infty)$ such that
\begin{align*}
\|\nabla u\|_{L^\infty(B_\rho(x_0))}^p\leq& \frac12 \|\nabla u\|_{L^\infty(B_\sigma(x_0))}^p+c\frac{\|H_\mu(|\nabla u|)^\frac{p}2\|_{L^1(B_\sigma(x_0))}^{\frac12\frac1{1-\gamma}}}{(\sigma-\rho)^{\frac1{1-\gamma}}}\\
&+ c\frac{\|H_\mu(|\nabla u|)^\frac{p}2\|_{L^1(B_\sigma(x_0))}}{(\sigma-\rho)^{2}}+c\mu^p.
\end{align*}
The claimed inequality \eqref{claim:L:apriorilipschitz} (for $B=B_1$) now follows from Lemma~\ref{L:holefilling}.
\end{proof}

\section{Proof of Theorem~\ref{T:1}}

Appealing to Lemma~\ref{L:apriorilipschitz}, the proof of Theorem~\ref{T:1} follows by approximation arguments. Below, we provide details following closely the scheme described in \cite[Section 8]{DM24}.

\begin{proof}[Proof of Theorem~\ref{T:1}]
Throughout the proof we write $\lesssim$ if $\leq$ holds up to a multiplicative constant that depends only on $\Lambda,\nu,p$ and $q$. We assume that $B_2\Subset\Omega$ and show
\begin{align}\label{est:apriorilipschitzu0}
\|\nabla u\|_{L^\infty(B_\frac12)}^p
\lesssim\biggl(\fint_{B_{1}}F(\nabla u)\,dx\biggr)^\frac1p+\biggl(\fint_{B_{1}}F(\nabla u)\,dx\biggr)^{\frac2{3p-q}}.
\end{align}
Clearly, the conclusion follows from a standard scaling, translation and covering arguments.\\

Let $\varphi$ be a non-negative, radially symmetric mollifier, i.e. it satisfies
$$
\varphi\geq0,\quad {\rm supp}\; \varphi\subset B_1,\quad \int_{\R^n}\varphi(x)\,dx=1,\quad \varphi(\cdot)=\widetilde \varphi(|\cdot|)\quad \mbox{for some $\widetilde\varphi\in C^\infty(\R)$}.
$$
We introduce two regularization parameters $\e,\delta\in(0,1)$, and set 
%We set  $\overline u_\e:=u\ast \varphi_{\e}$ with $\varphi_\e:=\e^{-2}\varphi(\frac{\cdot}\e)$ and $\varphi$ being a non-negative, radially symmetric mollifier, i.e. it satisfies
%%
%$$
%\varphi\geq0,\quad {\rm supp}\; \varphi\subset B_1,\quad \int_{\R^n}\varphi(x)\,dx=1,\quad \varphi(\cdot)=\widetilde \varphi(|\cdot|)\quad \mbox{for some $\widetilde\varphi\in C^\infty(\R)$}.
%$$
%Moreover, we set
%
\begin{equation}
\overline u_\e:=u\ast \varphi_\e,\quad F_{\e,\delta}(z):=F_\delta(z)+\sigma_\e (\mu+\delta+|z|^2)^{\frac{q}2},\quad F_\delta:=F\ast \varphi_\delta,\quad\sigma_\e:=\frac1{1+\e^{-1}+\|\nabla u_\e\|_{L^q(B_1)}^q}.
\end{equation}
The integrand $F_{\e,\delta}$ satisfies Assumption~\ref{ass:regF} for some $2\geq\mu_\delta>\mu>0$ satisfying $\mu_\delta\to\mu$ as $\delta\to0$, $L_{\e,\delta}\lesssim L$, $\nu\lesssim\nu_{\e,\delta}$ and $\sigma_\e\lesssim \tilde \nu_{\e,\delta}\lesssim \sigma_\e$ (see \cite[Section 8]{DM24}). Moreover, we note that $\partial F(z)\lesssim 1+|z|^q$ (see \eqref{est:partialF}) implies $|F_\delta(z)-F(z)|\lesssim \delta (1+|z|^q)$ and thus 
\begin{equation*}
\forall \e\in(0,1):\quad \limsup_{\delta\to0} \|F_\delta(\nabla \overline u_\e)-F(\nabla \overline u_\e)\|_{L^1(B_1)}=0.
\end{equation*}
In particular, we have
\begin{align}\label{limsupFepsueps}
\limsup_{\e\to0}\limsup_{\delta\to0} \int_{B_1}F_{\e,\delta}(\nabla \overline u_\e)\,dx\leq& \limsup_{\e\to0} \int_{B_1}F(\nabla \overline u_\e)+\frac{(\mu+|\nabla \overline u_\e|^2)^{\frac{q}2}}{1+\e^{-1}+\|\nabla \overline u_\e\|_{L^q(B_1)}^q}\,dx\notag\\
\leq& \limsup_{\e\to0}\int_{B_1}F(\nabla \overline u_\e)\,dx\notag\\
\leq&\limsup_{\e\to0}\int_{B_{1+\e}}F(\nabla u)\,dx\leq \int_{B_{1}}F(\nabla u)\,dx,
\end{align}
where we use in the last two inequality Jensen inequality in combination with $\nabla \overline u_\e=\nabla u\ast \varphi_\e$ and dominated convergence in combination with $F(\nabla u)\in L^1(B_2)$.
We denote by $u_{\e,\delta}\in W^{1,1}(B)$ the unique function satisfying
\begin{equation}\label{def:uepsdelta}
\int_{B_1}F_{\e,\delta}(\nabla u_{\e,\delta})\,dx\leq \int_{B_1}F_{\e,\delta}(\nabla v)\,dx\qquad\mbox{for all $v\in \overline u_\e+W_0^{1,1}(B)$.}
\end{equation}
The bound $H_\mu(z)^\frac{p}2\leq H_{\mu_\delta}(z)^\frac{p}2\lesssim F_{\e,\delta}(z)$ and the minimality of $u_{\e,\delta}$ yield
\begin{align}\label{L:apriorLp}
\int_{B_1}H_\mu(\nabla u_{\e,\delta})^\frac{p}2\,dx\leq \int_{B_1}H_{\mu_\delta}(\nabla u_{\e,\delta})^\frac{p}2\,dx\lesssim \int_{B_1}F_{\e,\delta}(\nabla u_{\e,\delta})\,dx\stackrel{\eqref{def:uepsdelta}}\leq \int_{B_1}F_{\e,\delta}(\nabla \overline u_\e)\,dx.
\end{align}
Since $F_{\e,\delta}$ satisfies Assumption~\ref{ass:regF} (with $\nu\lesssim\nu_{\e,\delta}$ and $L_{\e,\delta}\lesssim L$), we obtain by Lemma~\ref{L:apriorilipschitz} 
\begin{align}\label{est:apriorilipschitzum}
\|\nabla u_{\e,\delta}\|_{L^\infty(B_\frac12)}^p%\notag\\%\|G_T(\nabla u_{m})\|_{L^\infty(B_\frac12)}\notag\\
\stackrel{\eqref{claim:L:apriorilipschitz}}\lesssim& \int_{B_1} H_{\mu_\delta}(|\nabla u|)^\frac{p}2\,dx+\biggl(\int_{B_1} H_{\mu_\delta}(|\nabla u|)^\frac{p}2\,dx\biggr)^{\frac2{3p-q}}\notag\\\stackrel{\eqref{L:apriorLp}}\lesssim&\int_{B_{1}}F_{\e,\delta}(\nabla \overline u_\e)\,dx+\biggl(\int_{B_{1}}F_{\e,\delta}(\nabla \overline u_\e)\,dx\biggr)^{\frac2{3p-q}}.
\end{align}
Estimates \eqref{L:apriorLp} and \eqref{est:apriorilipschitzum} in combination with \eqref{limsupFepsueps} and $u_{\e,\delta}\in \overline u_\e+W_0^{1,1}(B_1)$ imply that there exist as sequence $(\delta_j)_{j\in\mathbb N}\subset (0,1)$ with $\delta_j\to0$ and $u_\e\in \overline u_\e+W_0^{1,1}(B)$ such that $u_{\e,\delta_j}\rightharpoonup u_\e$ in $W^{1,p}(B_1)$. Moreover, there exist $(\e_k)_{k\in\mathbb N}\subset(0,1)$ and $w\in u+W_0^{1,p}(B_1)$ such that $u_{\e_k}\rightharpoonup w$ in $W^{1,p}(B_1)$. The weak$^*$ lower semicontinuity of norms \eqref{est:apriorilipschitzum} and \eqref{limsupFepsueps} imply
\begin{align}\label{est:w}
\|\nabla w\|_{L^\infty(B_\frac12)}^p\leq\liminf_{k\to\infty}\liminf_{j\to\infty}\|\nabla u_{\e_k,\delta_j}\|_{L^\infty(B_\frac12)}^p%\notag\\%\|G_T(\nabla u_{m})\|_{L^\infty(B_\frac12)}\notag\\
\stackrel{\eqref{limsupFepsueps},\eqref{est:apriorilipschitzum}}\lesssim&\int_{B_{1}}F(\nabla  u)\,dx+\biggl(\int_{B_{1}}F(\nabla  u)\,dx\biggr)^{\frac2{3p-q}}
\end{align}
Finally, it remains to show that $u=w$. The convexity of $F$ and $\nabla u_{\e_k,\delta_j}\rightharpoonup \nabla u_{\e_k}$ imply for given $k\in\mathbb N$,
\begin{align*}
\int_{B_1}F(\nabla u_{\e_k})\,dx\leq& \liminf_{j\to\infty}\int_{B_1}F(\nabla u_{\e_k,\delta_j})\,dx \\
\leq&\liminf_{j\to\infty}\int_{B_1}F_{\e_k,\delta_j}(\nabla u_{\e_k,\delta_j})\,dx+\limsup_{j\to\infty}\int_{B_1}|F(\nabla u_{\e_k,\delta_j})-F_{\delta_j}(\nabla u_{\e_k,\delta_j})|\,dx\\
\leq&\liminf_{j\to\infty}\int_{B_1}F_{\e_k,\delta_j}(\nabla \overline u_{\e_k})\,dx,
\end{align*}
where we use in the last inequality the minimality of $u_{\e,\delta}$, $|F(z)-F_\delta(z)|\lesssim \delta(1+|z|^q)$ and the fact that by definition of $F_{\e,\delta}$, we have $\sup_j\|\nabla u_{\e_k,\delta_j}\|_{L^q(B_1)}^q<\infty$. Hence, sending $k\to\infty$, we have
\begin{align*}
\int_{B_1}F(\nabla w)\,dx\leq\liminf_{k\to\infty}\int_{B_1}F(\nabla u_{\e_k})\,dx \leq\limsup_{k\to\infty}\limsup_{j\to\infty}\int_{B_1}F_{\e_k,\delta_j}(\nabla \overline u_{\e_k})\,dx\stackrel{\eqref{limsupFepsueps}}\leq \int_{B_{1}}F(\nabla u)\,dx.
\end{align*}
The above inequality combined with $w\in u+W_0^{1,p}(B_1)$ and the strict convexity of $F$ implies $w=u$ and the desired estimate \eqref{est:apriorilipschitzu0} follows from \eqref{est:w}.
\end{proof}


\begin{thebibliography}{99}
%\bibitem{AF94} E.\ Acerbi and N.\ Fusco, Partial regularity under anisotropic $(p,q)$ growth conditions. \textit{J.\ Differential Equations} {\bf 107} (1994), 46--67.
%\bibitem{AT21} K.\ Adimurthi and V.\ Tewary, On Lipschitz regularity for bounded minimizers of functionals with (p,q) growth. arXiv:2108.06153 [math.AP].
\bibitem{BDS20} A.\ K.\ Balci, L.\ Diening and M.\ Surnachev, New examples on Lavrentiev gap using fractals. \textit{
Calc.\ Var.\ Partial Differential Equations} {\bf 59} (2020), no. 5, Paper No. 180, 34 pp.
%\bibitem{B16} P.\ Baroni, Riesz potential estimates for a general class of quasilinear equations. \textit{Calc.\ Var.\ Partial Differential Equations} 53 (2015), no. 3-4, 803--846.
\bibitem{BCM18} P.\ Baroni, M.\ Colombo and G.\ Mingione, Regularity for general functionals with double phase. \textit{Calc.\ Var.\ Partial Differential Equations} {\bf 57} (2018), no.~2, Art.~62.
%\bibitem{BB16} P.\ Bousquet and L.\ Brasco, Global Lipschitz continuity for minima of degenerate problems. \textit{Math.
%Ann.} {\bf 366} (2016), no. 3-4, 1403--1450.
\bibitem{BM20} L.\ Beck and G.\ Mingione, Lipschitz bounds and non-uniform ellipticity. \textit{Comm.\ Pure Appl.\ Math.} {\bf 73} (2020), 944--1034.
\bibitem{BS19a} P.\ Bella and M.\ Sch\"affner, Local Boundedness and Harnack Inequality for Solutions of Linear Nonuniformly Elliptic Equations. \textit{Comm.\ Pure Appl.\ Math.} {\bf 74} (2021), no.~3, 453--477.
%\bibitem{BS19b} P.\ Bella and M.\ Sch\"affner, Quenched invariance principle for random walks among random degenerate conductances. \textit{Ann.\ Probab.} {\bf 48} (2020), no.~1, 296--316.
\bibitem{BS19c} P.\ Bella and M.\ Sch\"affner, On the regularity of minimizers for scalar integral functionals with $(p,q)$-growth. \textit{Anal.\ PDE} {\bf 13} (2020), no.~7, 2241--2257.
\bibitem{BS24} P.\ Bella and M.\ Sch\"affner, Lipschitz bounds for integral functionals with $(p, q)$-growth conditions. \textit{Adv. Calc. Var.} (2022).
\bibitem{BR20} G.\ Bertazzoni and S.\ Ricc\'o, Lipschitz regularity results for a class of obstacle problems with nearly linear growth. \textit{J.\ Elliptic Parabol.\ Equ.} {\bf 6} (2020), no. 2, 883--918.
%\bibitem{BF} M.\ Bildhauer and M.\ Fuchs, Interior regularity for free and constrained local minimizers of variational integrals under general growth and ellipticity conditions. \textit{J. Math. Sci. (N.Y.)} {\bf 123} (2004), no. 6, 4565--4576.
\bibitem{BF03} M.\ Bildhauer and M.\ Fuchs, Twodimensional anisotropic variational problems. \textit{Calc. Var. Partial Differential Equations} {\bf 16} (2003), no. 2, 177--186.
%\bibitem{BF01} M.\ Bildhauer and M.\ Fuchs, Partial regularity for variational integrals with $(s,\mu,q)$-growth.
%\textit{Calc.\ Var.\ Partial Differential Equations} {\bf 13} (2001), no.~4, 537--560.
%\bibitem{BF03} M.\ Bildhauer and M.\ Fuchs, Twodimensional anisotropic variational problems. \textit{Calc.\ Var.\ Partial Differential Equations} {\bf 16} (2003), 177--186.
%\bibitem{B03} M.\ Bildhauer, \textit{Convex Variational Problems.} volume 1818 of \textit{Lecture Notes in Mathematics}. Springer, Berlin, 2003.
%\bibitem{Breit12} D.\ Breit, Dominic, New regularity theorems for non-autonomous variational integrals with (p,q)-growth. \textit{Calc. Var. Partial Differential Equations} {\bf 44} (2012), no.~1-2, 101--129. 
%\bibitem{CKP11} M.\ Carozza, J.\  Kristensen and A.\ Passarelli di Napoli, Higher differentiability of minimizers of convex variational integrals. \textit{Ann.\ Inst.\ H.\ Poincar\'e Anal.\ Non Lin\'eaire} {\bf 28} (2011), no.~3, 395--411.
\bibitem{BGS21} M.\ Bul\'i\v{c}ek, P.\ Gwiazda and J.\ Skrzeczkowski, On a range of exponents for absence of Lavrentiev phenomenon for double phase functionals. arXiv:2110.13945 [math.AP]
\bibitem{Borowski2023} M.\ Borowskia, I.\ Chlebicka, F.\ De Filippis, B.\ Miasojedowa, Absence and presence of Lavrentiev’s phenomenon for double phase functionals upon every choice of exponents. \textit{Calc. Var. Partial Differential Equations} {\bf 63} (2024), no. 2, 35.
%\bibitem{BO20} S.-S.\ Byun and J.\ Oh, Regularity results for generalized double phase functionals. \textit{Anal.\ PDE} {\bf 13} (2020), no. 5, 1269--1300. 
\bibitem{CKP14} M.\ Carozza, J.\  Kristensen and A.\ Passarelli di Napoli, Regularity of minimizers of autonomous convex variational integrals, \textit{Ann. Sc. Norm. Super. Pisa Cl. Sci. (5)} {\bf13} (2014), no.~4, 1065--1089.
%\bibitem{CKP11} M.\ Carozza, J.\  Kristensen and A.\ Passarelli di Napoli, Higher differentiability of minimizers of convex variational integrals, \textit{Ann.\ Inst.\ H.\ Poincar\'e Anal.\ Non Lin\'eaire} {\bf 28} (2011), no.~3, 395--411.
\bibitem{CS23} A.\ Cianchi and M.\ Sch\"affner, Local boundedness of minimizers under unbalanced Orlicz growth conditions. arXiv:2309.16803 [math.AP].
%\bibitem{Cianchi92} A.\ Cianchi, Maximizing the $L^\infty$-norm of the gradient of solutions to the Poisson equation, \textit{J.\ Geom.\ Anal.} {\bf 2}, 499--515, (1992).
%\bibitem{CianMaz11} A.\ Cianchi and V.\ G.\ Maz'ya, Global Lipschitz regularity for a class of quasilinear elliptic equations. \textit{Comm. Partial Differential Equations} {\bf 36}, 1, 100--133.
\bibitem{CFK20} I.\ Chlebicka, C.\ De Filippis, and L.\ Koch, Boundary regularity for manifold constrained $p(x)$-Harmonic maps. \textit{J. London Math.\ Soc.\ (2)} {\bf 104} (2021), 2335--2375.
%\bibitem{CM15} M.\ Colombo and G.\ Mingione, Regularity for double phase variational problems, \textit{Arch.\ Ration.\ Mech.\ Anal.} {\bf 215} (2015), no.~2, 443--496.
%\bibitem{CMM15} G. Cupini, P. Marcellini and E. Mascolo, Local boundedness of minimizers with limit growth conditions, \textit{J.\ Optimization Th. Appl.} {\bf 166} (2015), 1--22.
\bibitem{CMMP21} G.\ Cupini, P.\ Marcellini, E.\ Mascolo and A.\ Passarelli di Napoli, Lipschitz regularity for degenerate elliptic integrals with $p,q$-growth. \textit{Adv.\ Calc.\ Var.} {\bf 16} (2023), no. 2, 443--465.
%\bibitem{DM19} C.\ De Filippis and G.\ Mingione, On the regularity of minima of non-autonomous functionals, \textit{J.\ Geom.\ Anal.} {\bf 30} (2020), 1584--1626.
\bibitem{DKK} C.\ De Filippis, L.\ Koch and J.\ Kristensen, Quantified Legendreness and the regularity of minima. arXiv:2312.14659 [math.AP]. 
\bibitem{DM21} C.\ De Filippis and G.\ Mingione, Lipschitz bounds and nonautonomous integrals. \textit{Arch.\ Ration.\ Mech.\ Anal.} {\bf 242} (2021), 973--1057.
\bibitem{DM23b} C.\ De Filippis and G.\ Mingione, Regularity for double phase problems at nearly linear growth. \textit{Arch. Ration. Mech. Anal.} {\bf 247} (2023), no. 5, Paper No. 85, 50 pp.
\bibitem{DM23} C.\ De Filippis and G.\ Mingione, Nonuniformly elliptic Schauder Theory. \textit{Invent. Math.} {\bf 234} (2023), no. 3, 1109--1196. 
\bibitem{DM24} C.\ De Filippis and G.\ Mingione, The sharp growth rate in nonuniformly elliptic Schauder theory. arXiv:2401.07160 [math.AP]. 
\bibitem{DP24} F.\ De Filippis and M.\ Piccinini, Regularity for multi-phase problems at nearly linear growth. arXiv:2401.02186 [math.AP]
%\bibitem{Filippis19} C. De Filippis, C. Partial regularity for manifold constrained p(x)-harmonic maps. Calc. Var. 58, 47 (2019). 
%\bibitem{Filippis21} C.\ De Filippis, Quasiconvexity and partial regularity via nonlinear potentials. \textit{arXiv:2105.00503}.
\bibitem{EMM19} M.\ Eleuteri, P.\ Marcellini and E.\ Mascolo, Regularity for scalar integrals without structure conditions. \textit{Adv.\ Calc.\ Var.} {\bf 13} (2020), no. 3, 279--300. 
\bibitem{EMMP22} M.\ Eleuteri, P.\ Marcellini, E.\ Mascolo and S.\ Perrotta, Local Lipschitz continuity for energy integrals with slow growth. \textit{Ann. Mat. Pura Appl. (4)} {\bf 201} (2022), no. 3, 1005--1032. 
%\bibitem{EPT23} M.\ Eleuteri, S.\ Perrotta and G.\ Treu, Local Lipschitz continuity for energy integrals with slow growth and lower order terms. arXiv:2309.10727 [math.AP].
\bibitem{EP23} M.\ Eleuteri and A.\ Passarelli di Napoli, Lipschitz regularity for a priori bounded minimizers of integral functionals with nonstandard growth. arXiv:2310.04692 [math.AP]
\bibitem{ELM99} L.\ Esposito, F.\ Leonetti and G.\ Mingione, Higher integrability for minimizers of integral functionals with $(p,q)$ growth, \textit{J.\ Differential Equations} {\bf 157} (1999), no.~2, 414--438.
%\bibitem{ELM02} L.\ Esposito, F.\ Leonetti and G.\ Mingione, Regularity results for minimizers of irregular integrals with $(p,q)$ growth. \textit{Forum Math.} {\bf 14} (2002), no.~2, 245--272.
%\bibitem{ELM04} L.\ Esposito, F.\ Leonetti and G.\ Mingione, Sharp regularity for functionals with $(p,q)$ growth, \textit{J.\ Differential Equations} {\bf 204} (2004), no.~1, 5--55.
%\bibitem{Evans86} L.\ C.\ Evans, Quasiconvexity and partial regularity in the calculus of variations. \textit{Arch.\ Rational Mech.\ Anal.} {\bf 95} (1986), no.~3, 227--252. 
%\bibitem{FS93} N.\ Fusco and C.\ Sbordone, Some remarks on the regularity of minima of anisotropic integrals, \textit{Comm.\ Partial Differential Equations} {\bf 18} (1993), 153--167.
%\bibitem{G87} M.\ Giaquinta, Growth conditions and regularity, a counterexample, \textit{Manuscripta Math.} {\bf 59} (1987), no.~2, 245--248. 
\bibitem{FM} M.\ Fuchs and G.\ Mingione, Full $C^{1,\alpha}$-regularity for free and constrained local minimizers of elliptic variational integrals with nearly linear growth. \textit{Manuscripta Math.} {\bf 102} (2000), no. 2, 227--250.
\bibitem{GT} D.\ Gilbarg and N.\ Trudinger, \textit{Elliptic partial differential equations of second order}, Springer, 1998. 
\bibitem{Giu} E.\ Giusti, \textit{Direct methods in the calculus of variations,} World Scientific Publishing Co., Inc., River Edge, NJ, 2003. viii+403 pp
\bibitem{HO21} P.\ H\"ast\"o and J.\ Ok, Maximal regularity for local minimizer of non-autonomous functionals. \textit{J.\ Eur.\ Math.\ Soc.\ (JEMS)} {\bf 24} (2022), no. 4, 1285--1334.
% \bibitem{HHT17} P.\ Harjulehto, P.\ H\"ast\"o and O.\ Toivanen, H\"older regularity of quasiminimizers under generalized growth conditions. \textit{Calc.\ Var.\ Partial Differential Equations} {\bf 56} (2017), no.~2, Paper No.~22, 26 pp.
\bibitem{HS19} J.\ Hirsch and M.\ Sch\"affner, Growth conditions and regularity, an optimal local boundedness result. \textit{Commun.\ Contemp.\ Math.} {\bf 23} (2021), no.~3, Paper No.~2050029.
%\bibitem{HL} Q.\ Han and F.\ Lin, \textit{Elliptic partial differential equations}, Courant Lecture Notes in Mathematics, vol. 1, New York University, Courant Institute of Mathematical Sciences, New York. American Mathematical Society, Providence, RI (1997).
\bibitem{Koch} L.\ Koch, Global higher integrability for minimisers of convex functionals with $(p,q)$-growth. \textit{Calc.\ Var.\ Partial Differential Equations} {\bf 60} (2021), no. 2, Paper No. 63, 29 pp.
%\bibitem{KM13} T.\ Kuusi and G.\ Mingione, Linear potentials in nonlinear potential theory. \textit{Arch.\ Ration.\ Mech.\ Anal.} {\bf 207} (2013), no.~1, 215--246. 
%\bibitem{KM14} T.\ Kuusi and G.\ Mingione, A nonlinear Stein theorem. \textit{Calc.\ Var.\ Partial Differential Equations} {\bf 51} (2014), no. 1-2, 45--86.
%\bibitem{H92} M.\ C.\ Hong, Some remarks on the minimizers of variational integrals with nonstandard growth conditions, \textit{Boll. Un. Mat. Ital. A (7)} {\bf 6} (1992), no.~1, 91--101. 
%\bibitem{L93} G. M.\ Lieberman, The natural generalization of the natural conditions of Ladyzhenskaya and Uraltseva for elliptic equations, \textit{Comm.\ Partial Differential Equations} {\bf 16} (1991), 311--361.
%\bibitem{Mar89} P.\ Marcellini, Regularity of minimizers of integrals of the calculus of variations with nonstandard growth conditions, \textit{Arch.\ Rational Mech.\ Anal.} {\bf 105} (1989), no.~3, 267--284. 
\bibitem{Mar91} P.\ Marcellini, Regularity and existence of solutions of elliptic equations with $p,q$-growth conditions, \textit{J.\ Differential Equations} {\bf 90} (1991), no.~1, 1--30. 
\bibitem{Mar96}  P.\ Marcellini, Regularity for some scalar variational problems under general growth conditions. \textit{J.\ Optim.\ Theory Appl.} {\bf 90} (1996), no.~1, 161--181.
\bibitem{Mar21} P.\ Marcellini, Growth conditions and regularity for weak solutions to nonlinear elliptic pdes. \textit{J. Math. Anal. Appl.} {\bf 501} (2021), no.~1, Paper No. 124408, 32 pp. 
%\bibitem{Mar93} P.\ Marcellini, Regularity for elliptic equations with general growth conditions, \textit{J.\ Differential Equations}  {\bf 105} (1993), no.~3, 296-333.
\bibitem{Min24} G.\ Mingione, Nonlinear potential theoretic methods in nonuniformly ellliptic problems, arXiv:2402.02683 [math.AP].
%\bibitem{MaPa21} E.\ Mascolo, E and A. Passarelli di Napoli, Higher differentiability for a class of problems under p,q subquadratic growth. arXiv:2110.15874 [math.AP].
\bibitem{MR21} G.\ Mingione and V.\ R\v{a}dulescu, Recent developments in problems with nonstandard growth and nonuniform ellipticity. \textit{J.\ Math.\ Anal. Appl.} {\bf 501} (2021), no. 1, Paper No. 125197, 41 pp.
\bibitem{connor} C.\ Mooney, A proof of the {K}rylov-{S}afonov theorem without
              localization. \textit{Comm. Partial Differential Equations} {\bf 44} (2019), no~8, 681--690.
%\bibitem{PS} A.\ Passarelli Di Napoli and F.\ Siepe, A regularity result for a class of anisotropic systems. \textit{Rend.\ Ist.\ Mat.\ Univ.\ Trieste} (1996) {\bf 28}, no.~1-2, 13--31.
\bibitem{S21} M.\ Sch\"affner, Higher integrability for variational integrals with non-standard growth. \textit{Calc.\ Var.\ Partial Differential Equations} {\bf 60} (2021), no.~2, Paper No.~77, 16 pp.
%\bibitem{stein} E.\ M.\ Stein, Editor's note: the differentiability of functions in $\R^n$. \textit{Ann.\ of Math.} (2) 113 (1981), no. 2, 383--385.
%\bibitem{SY02} V.\ \v{S}ver\'ak and X.\ Yan, Non-Lipschitz minimizers of smooth uniformly convex functionals. \textit{Proc.\ Natl.\ Acad.\ Sci.\ USA} {\bf 99} (2002), 15269--15276. 
%\bibitem{TatarBook} L.\ Tartar, \textit{An introduction to Sobolev spaces and interpolation spaces}, Lecture Notes of the Unione Matematica Italiana, 3. Springer, Berlin; UMI, Bologna, 2007. 

%

%\bibitem{Moser60} J.\ Moser,  A new proof of De Giorgi's theorem concerning the regularity problem for elliptic differential equations, \textit{Comm.\ Pure Appl.\ Math.} {\bf 13} (1960), 457--468 .
%\bibitem{MS68} M.\ K.\ V.\ Murthy and G.\ Stampacchia, Boundary value problems for some degenerate-elliptic operators, \textit{Ann.\ Mat.\ Pura Appl.\ (4)} {\bf 80} (1968), 1--122. 
%\bibitem{T71} N.\ S.\ Trudinger, On the regularity of generalized solutions of linear, non-uniformly elliptic equations, \textit{Arch.\ Rational Mech.\ Anal.} {\bf 42} (1971), 50--62.
%\bibitem{T73} Trudinger, N.\ S. Linear elliptic operators with measurable coefficients, \textit{Ann.\ Scuola Norm.\ Sup.\ Pisa (3)} {\bf 27} (1973), 265--308.

\end{thebibliography}
\end{document}